\documentclass{amsart}
\usepackage{graphicx}
\usepackage{color}
\usepackage{subfig}
\usepackage[all]{xy}
\usepackage{amsfonts,amssymb,amscd,amsmath,enumerate,verbatim,newlfont,calc}

 \usepackage{amsfonts}

 \newtheorem{theorem}{Theorem}[section]

\newtheorem{proposition}[theorem]{Proposition}
\theoremstyle{definition}
\newtheorem{definition}[theorem]{Definition}

 \theoremstyle{remark}
\newtheorem{remark}[theorem]{Remark}

\numberwithin{equation}{section}
\begin{document}

 \title[Unified dynamical systems on coarse spaces   ]{Unified dynamical systems on coarse spaces  }
\author[ Najmeh Khajoei ]{{  Najmeh Khajoei
}\\{Mahani Mathematical Research Center Shahid Bahonar University of Kerman, Kerman, Iran.\\e-mails:  khajuee.najmeh@yahoo.com}}

\maketitle

\begin{abstract}
In this paper, we introduce a new class of dynamical systems on a coarse space with coarse time called, coarse dynamical system. The notion of coarse conjugacy is studied from coarse geometry point of view. Coarse orbits as
invariant sets under coarse conjugacy are deduced. It is Shown that the coproduct of two coarse dynamical systems is a coarse dynamical system. Finally, we define set-value coarse dynamical systems and prove if two coarse dynamical systems are coarse conjugate, then their corresponding set-value coarse dynamical systems are coarse conjugate.
\end{abstract}


\section{Introduction }
The coarse geometry is known as large-scale geometry which studies geometric objects from a large distance and large-scale features of them.
Historically, coarse geometry goes back from geometric group theory in \cite{ k9, k8, k10} and Mostow’s rigidity theorem and its generalization in \cite{k4, k5} and \cite{k6, k7} respectively. In \cite{k11} Roe considered coarse spaces as a large-scale counterpart of
uniform spaces and in \cite{k12} Nowak and Yu studied large-scale geometry of metric spaces.
 Large-scale geometry was studied on the group as well. The first attempt to study coarse groups was pioneered by Protasov in \cite{k19}, where this concept was introduced by using balleans. Also, in the same paper, he expressed coarse groups are uniquely obtained by a specific ideal of subsets of the group, which is called group ideal.
Recently coarse hyperspaces, i.e., structures induced on power sets, were inspired by metric and uniform hyperspace are studied in \cite{k21, k20}.\\
In coarse geometry, $\mathbb{Z}$ (integer number) and $\mathbb{R}$ (real number) are coarsely equivalent. This property motivates us to define a coarse time dynamical system on a coarse space in which the evolution time is a coarse group and in particular unified dynamical systems with coarse time $\mathbb{Z}$ or $\mathbb{R}$. Since from dynamical systems viewpoint, we encounter two classes of dynamical systems i.e., discrete and continuous dynamical systems our approach changes this aspect by unifying these two classes due to equivalency in coarse geometry. Therefore from this outlook, investigating features of dynamical systems does not depend on their time evolution. In the particular case, hyper semi-dynamical system is studied in \cite{mol}.\\
The paper is organized as follows. In Section $2$, we recall some needed preliminaries of the coarse geometry which include
basic definitions e.g., coarse space,  the product of coarse structures, subspace coarse structure, morphisms between coarse spaces, coarse group and asymptotic dimension. In Section $3$, a new class of dynamical system which is a coarse dynamical system with coarse time and in particular unified dynamical systems are presented. By defining conjugate coarse relation on coarse dynamical systems it is shown that
the coarse relation is an equivalence relation and preserves coarse-orbit. We prove the coproduct of two coarse dynamical systems is a coarse dynamical system. Also, we define set-value coarse dynamical systems and prove if two coarse dynamical systems are coarse conjugate, then their corresponding set-value coarse dynamical systems are coarse conjugate.

\section{Preliminaries}
In this section, some concepts in coarse geometry e.g., coarse space, coarse group and asymptotic dimension are presented. In the sequel, $\mathbb{N}, \mathbb{Z}$ and $\mathbb{R}$ denote the sets of natural numbers,
of integers and of real numbers, respectively. The following concepts are adapted from \cite{k13, k14, k11}.

\subsection{Coarse spaces} As mentioned in the previous section coarse geometry studies large-scale features
of spaces, neglecting their local, small-scale ones.  Two spaces are considered equivalent if they are the same as each other for a spectator whose point of view is getting farther and farther. In this subsection, we briefly provide some concepts which are needed in the sequel.

\begin{definition}
\cite{k11} A coarse space is a pair $(M,\mathcal{E}_M)$, where $M$ is a set and $\mathcal{E}_M\subseteq P(M\times M)$ a coarse structure on it, which means that\\
$(i)$ $\bigtriangleup_M:=\{(m,m) \mid m\in M\}\in \mathcal{E}_M$;\\
$(ii)$ $\mathcal{E}_M$ is closed under taking subsets;\\
$(iii)$ $\mathcal{E}_M$  is closed under finite unions;\\
$(iv)$ if $E\in \mathcal{E}_M$, $E^{-1}=\{(m_2, m_1) \mid (m_1,m_2)\in E\}\in \mathcal{E}_M$;\\
$(v)$ if $E, F\in \mathcal{E}_M$, then $E\circ F=\{( m_1, m_2)\in M\times M \mid\exists~ m_3\in M: (m_1, m_3)\in E, (m_3, m_2)\in F\}$.
\end{definition}
The properties $(ii)$ and $(iii)$ say that $\mathcal{E}_M$ is an ideal of subsets of $M\times M$ and an element $E$ of $\mathcal{E}_M$ is called entourage.
The subset $E[m]=\{m'\in M : (m,m')\in E\}$ is called the ball centred in $m$ with radius $E$. Furthermore, if $N\subseteq M$, then $E[N]:=\displaystyle\bigcup_{n\in N} E[n]$.\\
For a set $M$ a base of a coarse structure $\mathcal{B}$  is a family of entourages such that its completion
   $\mathcal{CL}(\mathcal{B}):=\{F\subseteq B \mid B\in \mathcal{B}\}$ is a coarse structure. Notice that $\mathcal{CL}(\mathcal{B})$ is the closure of
$\mathcal{B}$ under taking subsets.\\ Despite the fact that  for every set $M$ we can equip it with two coarse structure which are
the discrete coarse structures $\mathcal{E}_{M_{dis}}=\mathcal{CL}(\{\{\bigtriangleup_M\}\})$ and the
trivial (or indiscrete) coarse structure $\mathcal{E}_{M_{triv}}=P(M\times M) $.
There is a significant example of coarse structure for metric space which is called the metric-coarse structure. Consider an extended pseudo-metric $d$ on $M$, i.e., $d: M\times M\rightarrow \mathbb{R}^+\cup \{\infty\}$ with the following properties: \\
$(i)$ $d(m, m)=0$, for every $m\in M$;\\
$(ii)$ $d(m_1, m_2)=d(m_2, m_1)$, for every $m_1, m_2\in M$ (symmetry);\\
$(iii)$ $d(m_1, m_2)\leqslant d(m_1, m_3)+d(m_3, m_2)$; for every $m_1, m_2, m_3\in M$ (triangle inequality).\\
As usual we  make the assumption: $a+\infty=\infty+a=\infty$ and $a\leqslant \infty$ for all $a\in \mathbb{R}^+\cup\infty$.
For convenience we indicate $(M,d)$ as a metric space. For every $r>0$, by considering a specific subsets of the square $M\times M$, which is called the strip of width $r$, as follows:
$S_r=\displaystyle\bigcup_{m\in M} \big(\{m\}\times B(m,r)\big)$, where $B(m,r)$ signifies the open ball centred in $m$ with radius $r$.
The family $\mathcal{B}_d=\{S_r \mid r>0\}$ is a base of the metric coarse structure $\mathcal{E}_{M_{d}}=\mathcal{CL}(\mathcal{B}_d)$.\\
A coarse space $(M, \mathcal{E}_M)$ is metrisable if there exists an extended pseudo-metric $d$ on $M$ such that $\mathcal{E}_M=\mathcal{E}_{M_d}$.
\begin{definition}\cite{k13, k14}
Let $\{(M_i, \mathcal{E}_{M_i})\}_{i\in I}$ be a family of coarse spaces and $M=\prod_i M_i$ and $p_i: M\rightarrow M_i$, for every $i\in I$ be the projection maps.
Then the product coarse structure $\mathcal{E}=\prod_{i}\mathcal{E}_{M_{i}}$ is defined by the base
$ \mathcal{B}=\big\{\displaystyle\bigcap_{i\in I} (p_i\times p_i)^{-1}(E_i) \mid E_i\in \mathcal{E}_{M_i}, \forall i\in I \big\}$
\end{definition}
 Indeed, products and coproducts are two dual notions in category theory. In the next definition, we provide coarse structure on coproducts of coarse spaces.
\begin{definition}\label{n14}\cite{k1,k2}
Let $(M, \mathcal{E}_M)$ and $(\tilde{M}, \mathcal{E}_{\tilde{M}})$ be two coarse spaces. The coproduct coarse structure $\mathcal{E}$ on the disjoint union $M\sqcup\tilde{M}$ is defined by $\mathcal{E} :=\{(i_1\times i_1) (X)\cup(i_2\times i_2) (Y) \mid X\in \mathcal{E}_M, Y\in \mathcal{E}_{\tilde{Y}}\}$, where $i_1 : M\rightarrow M\sqcup \tilde{M}$ and $i_2 : \tilde{M}\rightarrow M\sqcup \tilde{M}$ are the canonical inclusions.
\end{definition}
One can easily extend this definition to the
coproduct of a finite number of coarse spaces. As for the infinite case, we refer to \cite{k24}.
\begin{definition}\cite{k11}
If $(M, \mathcal{E}_M)$ is a coarse space and $N\subset M$, then $N$ can be equipped with the subspace coarse structure $\mathcal{E}\mid_N =\{E\cap(N\times N) \mid E\in \mathcal{E}_M\}$. In this case, the pair $(Y, \mathcal{E}_M\mid_N)$ is called a coarse subspace.
\end{definition}
In order to define a map between two coarse spaces, the maps $f, g: K\rightarrow (M, \mathcal{E}_M)$ from a set $K$ to a coarse space $M$ are considered. We say $f$ and $g$ are close, and write $f\sim g$, if $\{(f(k), g(k)) \mid k\in K\}\in \mathcal{E}_M$. Furthermore, for the map $f:M\rightarrow N$ between sets, we denote by $f\times f: M\times M\rightarrow N\times N$ the map defined by $f\times f(m_1, m_2)=(f(m_1), f(m_2))$ for every $(m_1, m_2)\in M\times M$.
\begin{definition}\cite{ k13, k14,k2}\label{d1}
Let $(M, \mathcal{E}_M)$ and $(N, \mathcal{E}_N)$ be two coarse spaces. A map $f: M\rightarrow N$ is\\
$(i)$ bornologous if, for every $ E\in \mathcal{E}_M$, $(f\times f) (E)\in \mathcal{E}_N$;\\
$(ii)$ effectively proper if, for every $E\in \mathcal{E}_N$, $(f\times f)^{-1}(E)\in \mathcal{E}_M$;\\
$(iii)$ an asymorphism if one of the following equivalent properties is satisfied:\\
$\indent (iii_1)$ $f$ is bijective and both $f$ and $f^{-1}$ are bornologous,\\
 $\indent (iii_2)$ $f$ is bijective and both bornologous and effectively proper;\\
 $(iv)$ an asymorphic embedding if one of the following equivalent properties is satisfied:\\
 $\indent (iv_1)$ the restriction of $f$ to its image endowed with the subspace coarse structure is an asymorphism,\\
 $\indent(iv_2)$ $f$ is injective and both bornologous and effectively proper;\\
 $(v)$ a coarse equivalence if one of the following equivalent properties is satisfied:\\
 $\indent(v_1)$ $f$ is bornologous and there exists another bornologous map $g: N\rightarrow M$ (called coarse inverse) such that $g\circ f\sim Id_M$ and $f\circ g\sim Id_N$,\\
 $\indent(v_2)$ $f$ is bornologous and effectively proper and $f(M)$ is large in $N$.

\end{definition}
It is clear asymorphisms are precisely the bijective coarse equivalences.
We recall that a subset $A$ of coarse space $M$ is called large in $M$ if there exists $E\in\mathcal{E}_M$ such that $E[A]:= \displaystyle\bigcup_{a\in A} E(a)=M$ (see \cite{k15}). \\

\subsection{Coarse group}
We are interested in using groups with coarse structure as evolution time for dynamical system on coarse spaces, so some concepts of the coarse group in \cite{k13, k14, k2} are presented briefly. In fact, the coarse structures which are agreed with the algebraic structure of the groups are more remarkable.\\
 Let $G$ is a group and $g\in G$, the left-shift $f_g^l: G\rightarrow G$ and the right-shift $f_g^r: G\rightarrow G$ are defined as follows:\\
 for every $g'\in G$ $f_g^l(g')=gg'$ and $f_g^r(g')=g'g$.

 \begin{proposition}\label{n7} \cite{k13, k14} If $G$ be a group and  $\mathcal{E}_G$ be a coarse structure on it, then the following properties are equivalent:\\
 $(i)$ for every $E\in\mathcal{E}$, $GE:=\{(g e_1, ge_2) ~\vert~ g\in G, (e_1, e_2)\in E\}\in \mathcal{E}_G$;\\
 $(ii)$ the family $\mathcal{F}_g^l=\{f_g^l  ~\vert~  g\in G\}$ is uniformly bornologous, i.e., for every $E\in \mathcal{E}_G$ there exsits $F\in \mathcal{E}_G$ such that, for every $g\in G$, $(f_g^l\times f_g^l)(E)\subseteq F$.
 \end{proposition}
 \begin{definition}\cite{k13, k14}
  A coarse structure $\mathcal{E}_G$ on a group $G$ is said to be a left group coarse structure if it has the  equivalent properties from Proposition \ref{n7}. A left coarse group is a pair $(G, \mathcal{E}_G)$ of a group $G$ and a left group coarse structure $\mathcal{E}_G$ on $G$. Right group coarse structure and right coarse group can be defined in a similar way.
 \end{definition}

To clarify and make some examples of left(right) group coarse structures and (left) right coarse groups, the following definitions are needed.
\begin{definition}\label{n12}
\cite{k13, k14} Let $G$ be a group. A group ideal $\mathcal{I}$ is a family of subsets of $G$ containing the singleton $\{e\}$, where $e$ is the neutral element of G, such that:\\
$(i)$ $\mathcal{I}$ is an ideal;\\
$(ii)$ $HT:=\{ht~ \vert~ h\in H, t\in T\}$, for every $H, T\in \mathcal{I}$;\\
$(iii)$ $H^{-1}:=\{h^{-1}~\vert~ h\in H\}\in \mathcal{I}$.
\end{definition}
 $\bigcup \mathcal{I}$ is a subgroup of $G$ if $\mathcal{I}$ is a group ideal on $G$.
 \begin{definition}\label{n8}
 \cite{k13, k14}
 Let $G$ be a group and $\mathcal{I}$ be a group ideal. For every $H\in \mathcal{I}$, we define
 $$ E_H^l:= G(\{e\}\times H):= \displaystyle\bigcup_{g\in G} (\{g\}\times gH).$$
 \end{definition}
 The family $\mathcal{E}_{\mathcal{I}}^l:= \{ E\subseteq G\times G \mid \exists H\in\mathcal{I} : E\subseteq E_H^l\}$ is a left coarse group structure which is called left $\mathcal{I}$-group coarse structure, and the pair $(G, \mathcal{E}_G^l)$ is a left coarse group that is called, left $\mathcal{I}$-coarse group. \\
 Pay attention the family $\{E_H^l \mid H\in \mathcal{I}\}$ is a base of the $\mathcal{I}$-group coarse structure. In addition, for every $H\in\mathcal{I}$ and $g\in G$, $E_H^l[g]=g H$.\\
  Analogously, the right $\mathcal{I}$-group coarse structure $\mathcal{E}_{\mathcal{I}}^r$ is defined by the base $\{E_H^r \vert H\in \mathcal{I}\}$, where
 $$ E_H^l:= (\{e\})\times H)G:= \displaystyle\bigcup_{g\in G} (\{g\}\times Hg).$$
The following proposition indicates the left $\mathcal{I}$-group coarse structure and the right $\mathcal{I}$-group coarse structure are equivalent for every group $G$ and group ideal $\mathcal{I}$ on it.
 \begin{proposition}\cite{k13, k14}
 If $G$ be a group, $\mathcal{I}$ be a group ideal, and $j: G\rightarrow G$ defined by $j(g):=g^{-1}$, then $J: (G, \mathcal{E}_{\mathcal{I}}^l) \rightarrow (G, \mathcal{E}_{\mathcal{I}}^r)$ is an asymorphism.
 \end{proposition}
 The fundamental result in this subsection is the following proposition which indicates every left coarse group can be achieved as Definition \ref{n8}
 \begin{proposition}\cite{k16}\label{n9}
Let $G$ be a group and $\mathcal{E}_G$ be a coarse structure on it. Then the following properties are equivalent:\\
$(i)$ $(G, \mathcal{E}_G)$ is a left coarse group;\\
$(ii)$ $\mathcal{E}_G=\mathcal{E}_{\mathcal{I}}^l$ where $\mathcal{I}:=\{E[e] : E\in \mathcal{E}_G\}$.
 \end{proposition}
 In view of Proposition \ref{n9} coarse groups are equivalently determined by group ideals and this proposition satisfies for right coarse structures.
  In the sequel, left group coarse structures (and left coarse groups), are referred to as group coarse structures (and coarse groups).
  In the next ramark some examples of coarse groups are provided.
  \begin{remark}\cite{k13, k14}\label{n10}
  Let $G$ be a group. \\
  $(i)$ There is two trivial groups ideal $\{\{e\}\}$ and $P(G)$. $\{\{e\}\}$-group coarse structure is called  the discrete coarse structure, which is containing only the subsets of the diagonal and group ideal $P(G)$, creates the bounded coarse structure that means every subset of $G \times G$ is an entourage. We recall that, for a coarse space $(M, \mathcal{E}_M)$ and $m \in M$, a subset $N$ of $M$ is bounded from $m$ if there exists an entourage $E\in \mathcal{E}_M$ such that $N\subseteq E[m] $. A subset is called bounded  if it is bounded from a point.\\
  $(ii)$ The family $[G]^{<\omega} $ contain all finite subsets of $G$ is a group ideal. $[G]^{<\omega}$-coarse structure is called finitary-group coarse structure.\\
  $(iii)$ If $d$ be a left-invariant pseudo-metric on $G$, the family $\mathcal{B}_d:= \{H\subseteq G \mid\exists R\geqslant 0 : H\subseteq B_d(e, R)\}$ is a group ideal and $\mathcal{B}_d$-group coarse structure is called metric-group coarse structure. Note that $\mathcal{E}_{\mathcal{B}_d}=\mathcal{E}_{G{_d}}$.
  \end{remark}

\subsection{Asymptotic dimension}
One of the extremely significant coarse invariant is the asymptotic dimension that associates a value to coarse spaces which is invariant under coarse equivalences. It is the large scale
counterpart of the classical $\breve{C}$ech-Lebesgue covering dimension which was introduced by Gromov \cite{k17}.
We provide the necessary outline about the notion of asymptotic
dimension of coarse spaces.

\begin{definition}\cite{k13}
A family $\mathcal{U}$ of subsets of a coarse space $(M, \mathcal{E}_M)$ is uniformly bounded if there exists $E\in \mathcal{E}_M$ such that, for every $U\in \mathcal{U}$ and $u \in U, U\subseteq E[u]$.
\end{definition}

\begin{definition}\cite{k11}
A coarse space $(M, \mathcal{E}_M)$ has asymptotic dimension at most $n\in \mathbb{N}$, where $n\in \mathbb{N}$ and we write $asdim(M, \mathcal{E}_M) \leqslant n$, if, for every $E\in \mathcal{E}_M$ there exists a uniformly bounded cover $\mathcal{U}=\mathcal{U}_0\cup...\cup\mathcal{U}_n$ such that, for every $i =0, ..., n$ and every $U, V\in \mathcal{U}_i$, $E[U]\cap V=\varnothing$ (i.e., $\mathcal{U}_i$ is $E$-separated, for every $i =0, ..., n)$. If $asdim(M, \mathcal{E}_M) \leqslant n$ and $asdim(M, \mathcal{E}_M) >n -1$, we write $asdim(M, \mathcal{E}_M) =n$, where $n \in \mathbb{N}$. Finally, if, for every $n \in \mathbb{N}$, $asdim(M, \mathcal{E}_M) >n$, then $asdim(M, \mathcal{E}_M) =\infty$.
\end{definition}
There are the similar definitions for uniformly bounded and asymptotic dimension of coarse groups which are presented in the two following definitions.
\begin{definition}\cite{k13}
A family $\mathcal{U}$ of subsets of a coarse group $(G, \mathcal{E}_{\mathcal{I}})$ is uniformly bounded if and only if there exists $H\in \mathcal{I}$ such that $U\subseteq E _{H}[u] =uH$, for every $U\in \mathcal{U}$ and $u \in U$.
\end{definition}
It is said that $\mathcal{U}$ is $S$-disjoint if, for every pair of distinct elements $U, V\in \mathcal{U}$, $U\cap E_{S}[V] =U\cap VS=\varnothing$ where $S\in \mathcal{I}$ (see \cite{k13}).
\begin{definition}\cite{k13}
A coarse group $(G, \mathcal{E}_I)$ has asymptotic dimension at most $n$ (asdim$(G, \mathcal{E}_{\mathcal{I}}) \leqslant n$), where $n \in \mathbb{N}$, if, for every $S\in\mathcal{I}$, there exists a uniformly bounded cover $\mathcal{U}=\mathcal{U}_0\cup... \cup\mathcal{U}_n$ such that, for every $i =0, ..., n$, $\mathcal{U}_i$ is $S$-disjoint. The asymptotic dimension of $(G, \mathcal{E}_{\mathcal{I}})$ is $n$ if asdim$(G, \mathcal{E}_{\mathcal{I}}) \leqslant n$ and asdim$(G, \mathcal{E}_{\mathcal{I}}) >n -1$. Finally, asdim$(G,\mathcal{E}_\mathcal{I})=\infty$ if, for every $n \in \mathbb{N}$, asdim$(G, \mathcal{E}_{\mathcal{I}}) >n$.
\end{definition}

\begin{theorem}\label{n11}
\cite{k18}
Let $G$ be an abelian group. Then $asdim G = r_{0}(G)$,
where $G$ is endowed with the finitary-group coarse structure .
\end{theorem}
Let us recall that for an abelian group $G$, $r_0(G)$, is its free-rank which is the cardinality of the maximal
independent subset of $G$.

\section{Coarse time dynamical systems on coarse space }
In this section, we introduce dynamical systems on a coarse space with coarse time evolution. In coarse geometry, $\mathbb{Z}$  and $\mathbb{R}$ are coarsely equivalent groups. Because of this property
we call the dynamical system on a coarse space, with the time evolution of $\mathbb{Z}$ or $\mathbb{R}$, a unified dynamical system on coarse apace.
 We define coarse conjugacy and coarse orbits for this new class of dynamical system. Also, it is shown that the coproduct of two coarse dynamical systems is a coarse dynamical system.

\begin{definition}\label{n2}
A coarse dynamical system with coarse time is a triple \\$[(M,\mathcal{E}_M), \phi, (G,\mathcal{E}_G,\ast)]$, where $(M,\mathcal{E}_M)$ is a non-empty coarse space, $(G,\mathcal{E}_G,\ast)$ is a coarse group and $\phi$ is a set of asymorphism mappings $\phi^g: M\rightarrow M$ where $g\in G$ with the following properties:
\\
$(i)$ If $g_1, g_2 \in G$, and $m\in M$, then $\phi^{g_1}\circ\phi^{g_2}\in \phi^{g_1\ast g_2} (m)$, where $\phi^{g_1\ast g_2}(m)=\{ \phi^g(m): g\in g_1\ast g_2\}$; \\
$(ii)$ $\phi^e=Id_{M}$ where $e$ is the identity element of $G$.
\end{definition}
We call the mapping $\phi^g: M\rightarrow M$ a coarse evolution operator.
\begin{remark}
$(i)$ If $G=\mathbb{Z}$ or $G=\mathbb{R}$ we call a triple $[(M,\mathcal{E}_M), \phi, (G,\mathcal{E}_G,\ast)]$, unified dynamical system.\\
Let us clarify this denomination from coarse geometry point of view.
If we consider $\mathbb{Z}$ and $\mathbb{R}$ as a metric space with usual metric $d(a, b)=\vert a-b\mid$ is left-invariant, then $\mathbb{Z}$ and $\mathbb{R}$ are coarsely equivalent. It is enough to take the inclusion $i : \mathbb{Z}\hookrightarrow \mathbb{R}$ and the floor map
$ \lfloor .\rfloor:\mathbb{R}\rightarrow \mathbb{Z}$ such that, for every $r\in \mathbb{R}$, $\lfloor r\rfloor=max \{z\in \mathbb{Z} : z\leqslant r\}$. If we look at $\mathbb{Z}$ and $\mathbb{R}$ as a group with the binary operation $``+"$
 by the statement $(iii)$ of Remark \ref{n10} they are coarsely equivalent with their group-coarse structure.
 In fact, $\mathbb{Z}$ and $\mathbb{R}$ are similar to each other from the coarse geometry point of view.
One can imagine $\mathbb{Z}$ whose points get closer and closer until the whole space will simulate a line which is $\mathbb{R}$. \\
  We ask the reader to pay attention to this point if we consider the finitary-group coarse structure on $\mathbb{Z}$ and $\mathbb{R}$ the situation is very different, as the two coarse groups are not coarsely equivalent. Since by Theorem \ref{n11} $r_0(\mathbb{Z})=1$, while $r_0(\mathbb{R})=\infty$, and therefore they can not be coarsely equivalent.\\
  $(ii)$
Consider a coarse dynamical system $[(M,\mathcal{E}_M), \phi, (G, \mathcal{E}_G,\ast)]$. Let $N\subset M$ and $H\leqslant G$ i.e. $H$ is a subgroup of $G$. Then we can make a sub-coarse dynamical system $[(N,\mathcal{E}_N), \phi\mid_{N}, (H, \mathcal{E}_H,\ast)]$ by taking subspace coarse structure on $N$ and $H$.
\end{remark}

\begin{definition}\label{n1}
Two coarse dynamical systems $[(M,\mathcal{E}_M), \phi, (G, \mathcal{E}_G,\ast)]$ and  \\$[(\tilde{M},\mathcal{E}_{\tilde{M}}), \tilde{\phi}, (\tilde{G}, \mathcal{E}_{\tilde{G}},\tilde{\ast})]$ are called $(f,h)$-coarse conjugate if there exist an asymorphism $f: (M, \mathcal{E}_M)\rightarrow (\tilde{M}, \mathcal{E}_{\tilde{M}})$ between coarse spaces and an asymorphism $h: (G, \mathcal{E}_G)\rightarrow (\tilde{G}, \mathcal{E}_{\tilde{G}}) $ such that the following two axioms hold.\\
$(i)$ $h(g_1\ast g_2)=h(g_1)\tilde{\ast} h(g_2)$ for all $g_1, g_2\in G$;\\
$(ii)$ $f\circ \phi^g=\tilde{\phi}^{h(g)}\circ f$ for all $g\in G$ (see the following diagram).
\begin{displaymath}
\xymatrix{
M \ar[r]^{\phi^g} \ar[d]_f &
M \ar[d]^{f} \\
\tilde{M} \ar[r]_{\tilde{\phi}^{h(g)}} & \tilde{M} }
\end{displaymath}
\end{definition}
The next theorem implies that conjugate coarse relation is an equivalence relation on coarse dynamical systems.
\begin{theorem}
Let $(f, h)$ be a coarse conjugate relation between $[(M,\mathcal{E}_M), \phi, (G, \mathcal{E}_G,\ast)]$ and  $[(\tilde{M},\mathcal{E}_{\tilde{M}}), \tilde{\phi}, (\tilde{G}, \mathcal{E}_{\tilde{G}},\tilde{\ast})]$, and let $(k, l)$ be a coarse conjugate relation between $[(\tilde{M},\mathcal{E}_{\tilde{M}}), \tilde{\phi}, (\tilde{G}, \mathcal{E}_{\tilde{G}},\tilde{\ast})]$ and  $[(\bar{M},\mathcal{E}_{\bar{M}}), \bar{\phi}, (\bar{G}, \mathcal{E}_{\bar{G}},\bar{\ast})]$. Then\\
$(i)$ $(f^{-1}, h^{-1})$ is a coarse conjugate relation between $[(\tilde{M},\mathcal{E}_{\tilde{M}}), \tilde{\phi}, (\tilde{G}, \mathcal{E}_{\tilde{G}},\tilde{\ast})]$ and  $[(M,\mathcal{E}_M), \phi, (G, \mathcal{E}_G,\ast)] $.\\
$(ii)$ $(k\circ f, l\circ h)$ is a coarse conjugate relation between $[(M,\mathcal{E}_M), \phi, (G, \mathcal{E}_G,\ast)] $ and $[(\tilde{M},\mathcal{E}_{\tilde{M}}), \tilde{\phi}, (\tilde{G}, \mathcal{E}_{\tilde{G}},\tilde{\ast})]$.
\end{theorem}
\begin{proof}
Proof (i). Since $f$  and $h$ are asymorphisms. It is enough to prove two properties in Definition \ref{n1} are satisfied.\\ If $\tilde{g}, \tilde{s}\in\tilde{G}$, then
$h^{-1}(\tilde{g}\title{\ast}\tilde{s})=h^{-1}(h(g\ast s))=h^{-1} (h(g)\tilde{\ast}h(s))=h^{-1}(h(g))\ast h^{-1}(h(s))=h^{-1}(\tilde{g})\ast h^{-1}(\tilde{s})$.\\
For $\tilde{g}\in\tilde{G}$ we have $f\circ\phi^{h^{-1}(\tilde{g})}\circ f^{-1}=\tilde{\phi}^{h\circ h^{-1}(\tilde{g})}\circ f\circ f^{-1}=\tilde{\phi}^{\tilde{g}}$. Therefore $\phi^{h^{-1}(\tilde{g})}\circ f^{-1}=f^{-1}\circ\tilde{\phi}^{\tilde{g}}$.\\
Proof (ii). $ l\circ h(g_1\ast g_2)=l(h(g_1\ast g_2))=l(h(g_1)\tilde{\ast} h(g_2))=l\circ h(g_1) \bar{\ast} l\circ h(g_2)$ for $g_1, g_2 \in G$ and\\
$\bar{\phi}^{l\circ h(g)}\circ (k\circ f)=k\circ\tilde{\phi}^{h(g)}\circ f=k\circ (f\circ\phi^g)=(k\circ f)\phi^g$.
\end{proof}

\begin{definition}
The set $O^G(m)=\{\phi^g(m): g\in G\}$ is called coarse-orbit for $m\in M$.
\end{definition}
It is clear $O^{G(m)}$ is a coarse group. Because for any $\phi^{g}(m), \phi^{\tilde{g}}(m)\in O^{G}(m)$ we have $\phi^{g}(m) \ast \phi^{{\tilde{g}}^{-1}}(m)=\phi^{g\ast\tilde{g}^{-1}}(m)\in O^G(m)$ where $g, \tilde{g}\in G$. So as we mentioned in Section $(2)$ it can be equipped to a coarse structure.\\
In the next theorem we show the coarse conjugate relation preserves coarse-orbit.
\begin{theorem}
If $[(M,\mathcal{E}_M), \phi, (G, \mathcal{E}_G,\ast)]$ and  $[(\tilde{M},\mathcal{E}_{\tilde{M}}), \tilde{\phi}, (\tilde{G}, \mathcal{E}_{\tilde{G}},\tilde{\ast})]$ are coarse conjugate under a coarse conjugate $(f,h)$, then $f(O^G(m))=O^{\tilde{G}}(f(m))$.
\end{theorem}
\begin{proof}
Let $s\in f(O^G(m))$. Then $s=f(\phi^g(m))=(\tilde{\phi}^{h(g)}\circ f)(m)=\tilde{\phi}^{h(g)}(f(m))\in O^{\tilde{G}}(f(m))$. So $f(O^G(m))\subseteq O^{\tilde{G}}(f(m))$. Because coarse conjugate relation is an equivalence relation. One can easily check $f^{-1}(O^{\tilde{G}}(f(m))\subseteq O^G(m)$. Therefore $O^{\tilde{G}}(f(m)\subseteq f(O^G(m))$.
\end{proof}

\begin{definition}
If $\phi_1: M_1\rightarrow \tilde{M}_1$ and $\phi_2: M_2\rightarrow \tilde{M}_2$ are two maps, a map between disjoint unions can be defined as follows: $\phi_1\sqcup\phi_2: M_1\sqcup M_2\rightarrow \tilde{M}_1\sqcup \tilde{M}_2$ for every $i_j(m)\in M_1\sqcup M_2$, $\phi_1 \sqcup\phi_2(i_j(m)):=i_j(\phi_j(m))$ for $j=1, 2$.
\end{definition}
In the sequel, the members of the disjoint union of $M$ and $\tilde{M}$ are considered as follows: $M \sqcup\tilde{M}=\{\big((1,m) (2,\tilde{m})\big) : m\in M, \tilde{m}\in\tilde{M}\}$.
\begin{remark}
If $(G,\ast)$ and $(\tilde{G},\tilde{\ast})$ be two groups. Then $(G\sqcup \tilde{G},\ast\sqcup\tilde{\ast})$  is group
where $\big( (1,g_1) (2,\tilde{g}_1)\big)\ast\sqcup\tilde{\ast} \big( (1,g_2) (2,\tilde{g}_2)\big):= \big((1,g_1\ast g_2)(2,\tilde{g}_1\tilde{\ast}\tilde{g}_2)\big)$ for all \\
$\big( (1,g_1) (2,\tilde{g}_1)\big), \big( (1,g_2) (2,\tilde{g}_2)\big)\in G\sqcup \tilde{G}$.
Because $\big( (1,g_1) (2,\tilde{g}_1)\big)\ast\sqcup\tilde{\ast} \big( (1,g_2^{-1}) (2,\tilde{g}_2^{-1})\big)= \big((1,g_1\ast g_2^{-1})(2,\tilde{g}_1\tilde{\ast}\tilde{g}_2^{-1})\big)\in  G\sqcup \tilde{G}$.
\end{remark}

\begin{theorem}\label{n15}
Let $[(M,\mathcal{E}_M), \phi, (G, \mathcal{E}_G,\ast)]$ and  $[(\tilde{M},\mathcal{E}_{\tilde{M}}), \tilde{\phi}, (\tilde{G}, \mathcal{E}_{\tilde{G}},\tilde{\ast})]$
 be two coarse dynamical systems. Then $[(M \sqcup\tilde{M},\mathcal{E}), \phi \sqcup\tilde{\phi}, (G\sqcup \tilde{G},\ast\sqcup\tilde{\ast})]$ is a coarse dynamical system.
\end{theorem}
\begin{proof}
In view of Definition \ref{n14} both $M \sqcup\tilde{M}$ and $G\sqcup \tilde{G}$ are equipped by coarse structure. So we show two properties in Definition \ref{n2} are satisfied.\\
$(i)$ $(\phi^{g_1} \sqcup\tilde{\phi}^{\tilde{g}_1})\circ(\phi^{g_2}\sqcup\tilde{\phi}^{\tilde{g}_2})\big((1,m)(2,\tilde{m})\big)=(\phi^{g_1}\circ\phi^{g_2})(1,m)\sqcup(\tilde{\phi}^{\tilde{g}_1}\circ\tilde{\phi}^{\tilde{g}_2})(2,\tilde{m})\in\phi^{g_1\ast g_2}(1,m)\sqcup\tilde{\phi}^{\tilde{g}_1\tilde{\ast}\tilde{g}_2}(2,\tilde{m})$ , where
$\phi^{g_1\ast g_2}(1,m)\sqcup\tilde{\phi}^{\tilde{g}_1\tilde{\ast}\tilde{g}_2}(2,\tilde{m})=\{\big((1,\phi^s(m)) (2,\tilde{\phi}^t(\tilde{m}) )\big)\mid s\in g_1\ast g_2, t\in\tilde{g}_1\tilde{\ast}\tilde{g}_2\}$
if $g_1, g_2\in G$, $\tilde{g}_1$, $\tilde{g}_2\in \tilde{G}$, $m\in M$ and $\tilde{m}\in\tilde{M}$.\\
$(ii)$ $\phi^e\sqcup\tilde{\phi}^{\tilde{e}}\big( (1,m)(2,\tilde{m})  \big)=\big((1,\phi^e(m))(2,\tilde{\phi}^{\tilde{e}}(\tilde{m})\big)=\big((1,Id_M)(2,Id_{\tilde{M}})\big)=Id_M\sqcup Id_{\tilde{M}}$.
\end{proof}

\begin{theorem}\label{n16}
If $(f, h)$ be a coarse conjugate relation between $[(M,\mathcal{E}_M), \phi, (G, \mathcal{E}_G,\ast)]$ and  $[(\tilde{M},\mathcal{E}_{\tilde{M}}), \tilde{\phi}, (\tilde{G}, \mathcal{E}_{\tilde{G}},\tilde{\ast})]$, and let $(k, l)$ be a coarse conjugate relation between $[(\hat{M},\mathcal{E}_{\hat{M}}), \hat{\phi}, (\hat{G}, \mathcal{E}_{\hat{G}},\hat{\ast})]$ and  $[(\bar{M},\mathcal{E}_{\bar{M}}), \bar{\phi}, (\bar{G}, \mathcal{E}_{\bar{G}},\bar{\ast})]$, then two coarse dynamical systems $[M\sqcup \hat{M},\phi\sqcup\hat{\phi}, (G\sqcup \hat{G}, \ast\sqcup\hat{\ast})]$ and $[\tilde{M}\sqcup\bar{M}, \tilde{\phi}\sqcup\bar{\phi}, (\tilde{G}\sqcup\bar{G},\tilde{\ast}\sqcup\bar{\ast})]$ are $(f\sqcup k, h\sqcup l)$ coarse conjugate dynamical systems.
\end{theorem}
\begin{proof}
Firstly, since $f$ and $k$ are asymorphisms. Then $f^{-1}\sqcup k^{-1}$ is bornologous map and $(f\sqcup k)\circ(f^{-1}\sqcup k^{-1})=(Id_M\sqcup Id_{\tilde{M}})$,  $(f^{-1}\sqcup k^{-1})\circ(f\sqcup k)=(Id_{\tilde{M}}\sqcup Id_M )$.\\ 
Secondly, it is clear $h\sqcup l$ is an asymorphism from $ G\sqcup\hat{G}$ to $\tilde{G}\sqcup\bar{G}$.\\
Thirdly, $h\sqcup l[\big((1,g_1)(2,\hat{g}_1)\big) \ast\sqcup\hat{\ast}\big((1,g_2)(2,\hat{g})\big)]=h(g_1)\tilde{\ast} h(g_2) \sqcup l(\hat{g}_2)\bar{\ast}l(\hat{g}_2)=h(g_1)\sqcup l(\hat{g}_1) \tilde{\ast}\sqcup\bar{\ast}h(g_2)\sqcup l(\hat{g}_2)=h\sqcup l[\big((1,g_1)(2,\hat{g}_1)\big)]\tilde{\ast}\sqcup\bar{\ast}h\sqcup l[\big((1,g_2),(2,\hat{g}_2)\big)]$.\\
Finally, to complete the proof we show the following diagram is commutative.
\begin{displaymath}
\xymatrix{
M\sqcup\hat{M} \ar[r]^{\phi^g\sqcup \hat{\phi}^{\hat{g}}} \ar[d]_{f\sqcup k} &
M\sqcup\hat{M} \ar[d]^{f\sqcup k} \\
\tilde{M}\sqcup\bar{M} \ar[r]_{\tilde{\phi}^{\tilde{g}}\sqcup \bar{\phi}^{\bar{g}}} & \tilde{M}\sqcup\bar{M} }
\end{displaymath}
So we have
$(f\sqcup k)\circ(\phi^g\sqcup \hat{\phi}^{\hat{g}})=(f\circ\phi^g)\sqcup(k\circ\hat{\phi}^{\hat{g}})=(\tilde{\phi}^{\tilde{g}}\circ f)\sqcup(\bar{\phi}^{\bar{g}}\circ k)=(\tilde{\phi}^{\tilde{g}}\sqcup \bar{\phi}^{\bar{g}})\circ(f\sqcup k)$.
\end{proof}
\begin{remark}
The Theorems \ref{n15} and \ref{n16} are satisfied in the case of considering product instead of coproduct of coarse dynamical systems.
We leave it to the reader to prove similar results.
\end{remark}

We are going to define a set-value coarse dynamical system. So some concepts of coarse structure on the power set of the coarse space $(M, \mathcal{E}_M)$ are provided.
The following definition tells us how to make a coarse structure on $P(M)$ i.e., the power set of the coarse space $(M, \mathcal{E}_M)$.
\begin{definition}\cite{k2}
Let $M$ be a set and $E$ be an entourge of $M$. Consider $E^*:=\{(K, L)\in P(M)\times P(M) \mid K\subseteq E[L], L\subseteq E[K]\}.$ For coarse space $(M,\mathcal{E}_M)$, the family $B_{\mathcal{E}_M}^*=\{X^* : X\in\mathcal{E}_M\}$ creates a base of coarse structure $\exp \mathcal{E}_M$ on $P(M)$ and $(P(M), \exp \mathcal{E}_M)$ is called coarse hyperspace.
\end{definition}
 \begin{proposition}\cite{k2}
 If $(M, \mathcal{E}_M)$ is a coarse space, then $\exp \mathcal{E}_M$ is a coarse structure.
\end{proposition}
\begin{proposition}\label{n17}\cite{k2}
Let $f: (M, \mathcal{E}_M) \rightarrow (\tilde{M},\mathcal{E}_{\tilde{M}})$ be a map between coarse spaces. Then\\
$(i)$ $f$ is bornologous if and only if $\exp f$ bornologous;\\
$(ii)$ $f$ effectively proper if and only if $\exp f$ effectively proper;\\
$(iii)$ $f$ an asymorphism if and only if $\exp f$ is an asymorphism;\\
$(iv)$ $f$ a coarse equivalence if and only if $\exp f$ is a coarse equivalence.
\end{proposition}
\begin{remark}
Due to Proposition \ref{n17}. In Definition \ref{n2}  the asymorphism mapping $\phi^g: (M, \mathcal{E}_M)\rightarrow  (M, \mathcal{E}_M)$ between coarse spaces induces the asymorphism mapping $\exp \phi^g: (P(M), \exp \mathcal{E}_M)\rightarrow (P(M), \exp \mathcal{E}_M)$. So we can define the set-value coarse dynamical system.
\end{remark}

\begin{definition}
Consider a coarse dynamical system $[(M,\mathcal{E}_M), \phi, (G,\mathcal{E}_G,\ast)]$. \\A set-value coarse dynamical system with caorse time which is induced by\\
 $[(M,\mathcal{E}_M), \phi, (G,\mathcal{E}_G,\ast)]$ is a triple $[(P(M),\exp \mathcal{E}_M), \exp\phi, (G,\mathcal{E}_G,\ast)]$ with the following properties:
\\
$(i)$ If $g_1, g_2 \in G$, and $N\in P(M)$, then $\exp\phi^{g_1}\circ\exp\phi^{g_2}\in \exp\phi^{g_1\ast g_2} (N)$, where $\exp\phi^{g_1\ast g_2}(N)=\{ \exp\phi^g(N): g\in g_1\ast g_2\}$; \\
$(ii)$ $\exp\phi^e=Id_{P(M)}$ where $e$ is the identity element of $G$.
\end{definition}
The mapping $\exp\phi^g: P(M)\rightarrow P(M)$ is called a set-value coarse evolution operator.

\begin{theorem}
Let two coarse dynamical systems $[(M,\mathcal{E}_M), \phi, (G, \mathcal{E}_G,\ast)]$ and  \\$[(\tilde{M},\mathcal{E}_{\tilde{M}}), \tilde{\phi}, (\tilde{G}, \mathcal{E}_{\tilde{G}},\tilde{\ast})]$ be $(f,h)$-coarse conjugate.
Then $[(P(M),\exp \mathcal{E}_M), \phi, (G, \mathcal{E}_G,\ast)]$ and $[(P(\tilde{M}),\exp\mathcal{E}_{\tilde{M}}), \tilde{\phi}, (\tilde{G}, \mathcal{E}_{\tilde{G}},\tilde{\ast})]$ are $(\exp f, h)$ coarse conjugate.
\end{theorem}
\begin{proof}
In view of Proposition \ref{n17} the map $f$ is an asymorphism if and only if $\exp f$ is an asymorphism and $\phi^g, \tilde{\phi}^{h(g)}$ are asymorphism if and only if $\exp \phi^g, \exp \tilde{\phi}^{h(g)}$ are asymorphism.
It is clear that  $f\circ \phi^g=\tilde{\phi}^{h(g)}\circ f$ if and only if $\exp f\circ \exp\phi^g=\exp\tilde{\phi}^{h(g)}\circ \exp f$ for all $g\in G$.
\end{proof}


\end{document}